\documentclass[12pt]{amsart}
\usepackage[utf8]{inputenc}
\usepackage[T1]{fontenc}
\usepackage{amsmath,amsthm,amssymb}
\usepackage[shortlabels]{enumitem}
\usepackage{xcolor}

\newtheorem{theorem}{Theorem}[section]
\newtheorem{remark}[theorem]{Remark}

\newtheorem{claim}[theorem]{Claim}
\newtheorem*{conjecture*}{Conjecture}

\newtheorem*{acknowledgement}{Acknowledgements}

\newtheorem{lemma}[theorem]{Lemma}
\newtheorem{prop}[theorem]{Proposition}

\newcommand{\TT}{\mathbb{T}}
\newcommand{\WW}{\mathbb{W}}

\title[Weighted decoupling and Falconer distance problem]{Weighted refined decoupling estimates and application to Falconer distance set problem}
\author[X. Du, Y. Ou, K. Ren and R. Zhang]{Xiumin Du, Yumeng Ou, Kevin Ren, and Ruixiang Zhang}

\newcommand{\R}{\mathbb{R}}

\newcommand{\lesim}{\lesssim}
\newcommand{\eps}{\varepsilon}

\newcommand{\norm}[1]{\| #1 \|}

\DeclareMathOperator{\dist}{dist}
\newcommand{\cE}{\mathcal{E}}

\newcommand{\al}{\alpha}

\newcommand{\ka}{\kappa}

\newcommand{\La}{\Lambda}

\newcommand{\Om}{\Omega}

\newcommand{\ZZ}{\mathbb{Z}}

\newcommand{\hichi}{\raisebox{0.7ex}{\(\chi\)}}

\begin{document}

\maketitle

\begin{abstract}
We prove some weighted refined decoupling estimates. As an application, we give an alternative proof of the following result on Falconer's distance set problem by the authors in a companion work: if a compact set $E\subset \mathbb{R}^d$ has Hausdorff dimension larger than $\frac{d}{2}+\frac{1}{4}-\frac{1}{8d+4}$, where $d\geq 4$, then there is a point $x\in E$ such that the pinned distance set $\Delta_x(E)$ has positive Lebesgue measure. Aside from this application, the weighted refined decoupling estimates may be of independent interest.  
\end{abstract}

\section{Introduction}
In this paper, we prove some weighted refined decoupling estimates (see Theorems \ref{thm-RD-lp-0} and \ref{thm-RD-lp}) and discuss their application to Falconer's distance set problem.

\subsection{Weighted refined decoupling estimates}
Here is the setup for refined decoupling estimates.

Suppose that $S \subset \mathbb{R}^d$ is a compact and strictly convex $C^2$ hypersurface with Gaussian curvature $\sim 1$. 

For any $\epsilon>0$, suppose there exists $0<\beta\ll \epsilon$ satisfying the following.   Suppose that the $R^{-1}$-neighborhood of $S$ is partitioned into $R^{-1/2} \times ... \times R^{-1/2} \times R^{-1}$ blocks $\theta$.  For each $\theta$, let $\mathbb{T}_\theta$ be a set of finitely overlapping tubes of dimensions $R^{1/2 + \beta} \times \cdots \times R^{1/2+\beta}\times R$ with long axis perpendicular to $\theta$, let $G(\theta)\in \mathbb{S}^{d-1}$ denote this direction, and let $\mathbb{T} = \cup_\theta \mathbb{T}_\theta$. Each $T\in \mathbb T$ belongs to $\mathbb{T}_{\theta}$ for a single $\theta$, and we let $\theta(T)$ denote
this $\theta$. We say that $f$ is microlocalized to $(T,\theta(T))$ if $f$ is essentially supported in
$2T$ and $\hat{f}$ is essentially supported in $2\theta(T)$. 

Here is our first main result on weighted refined decoupling estimates.

\begin{theorem} \label{thm-RD-lp-0}
 Suppose that $f=\sum_{T\in \mathbb{W}}f_T$, where $\mathbb W\subset \mathbb T$ and each $f_T$ is microlocalized to $(T,\theta(T))$. Let $Y$ be a union of $R^{1/2}$-cubes in $B^d_R$ each of which intersects at most $M$ tubes $T\in \mathbb W$. Denote $p_d=\frac{2(d+1)}{d-1}$. Then the following refined decoupling inequalities hold.

\noindent (a) Let $p\geq 2$. Then
\begin{equation}\label{eq-RD-lp}
    \|f\|_{L^p(Y)} \lesssim_\epsilon R^{\gamma_d(p)+\epsilon} M^{\frac 12 -\frac 1p} \left(\sum_{T\in \mathbb W}\|f_T\|_{L^p}^p\right)^{\frac 1p},
\end{equation}
where $\gamma_d(p)=0$ when $2\leq p\leq p_d$, and $\gamma_d(p)=\frac{d-1}{4}-\frac{d+1}{2p}$ when $p\geq p_d$.

\noindent (b) Let $p\leq p_d$, $\alpha\leq d$. Let $H:Y\to [0,1]$ be a function satisfying that $\int_Q H(x)\,dx \lesssim R^{\alpha/2}$ for any $R^{1/2}$-cube $Q$ in $Y$. Then
\begin{equation}\label{eq-RD-lp-R-alpha-0}
    \|f\|_{L^p(Y;Hdx)} \lesssim_\epsilon R^{\frac 12 (\alpha -d)(\frac 1p -\frac{1}{p_d})+\epsilon} M^{\frac 12 -\frac 1p} \left(\sum_{T\in \mathbb W}\|f_T\|_{L^p}^p\right)^{\frac 1p}.
\end{equation}

\end{theorem}

The study of inequalities of this type originated from the work \cite{DGL}, where linear and bilinear refined Strichartz estimates are established and applied to resolve Carleson's pointwise convergence problem of Schr\"odinger solutions in dimension $2+1$. Later, multilinear refined Strichartz estimates are proved in \cite{DGLZ} and play an important role in the final resolution of the pointwise convergence problem in all dimensions \cite{du2019sharp}. Refined decoupling inequalities are stronger versions of linear refined Strichartz estimates; they first appeared in \cite{guth2020falconer} and played a key role in recent study of the Falconer distance set problem \cite{guth2020falconer, du2021improved}. See \cite{guth2020falconer} for a comparison between Bourgain--Demeter's decoupling theorem \cite{BDdecoupling} and the refined decoupling theorem.

In the case that $2\leq p\leq p_d$, Theorem \ref{thm-RD-lp-0}(a) and the refined decoupling theorem in
\cite[Theorem 4.2]{guth2020falconer} are equivalent. The main novelty is Theorem \ref{thm-RD-lp-0}(b), which says that if we take the weighted $L^p$-norm $\|f\|_{L^p(Y; Hdx)}$ on the left-hand side of the decoupling inequality, where $H$ is a weight that satisfies the ball condition with exponent $\alpha$ at scale $R^{1/2}$, then we have an extra gain $R^{\frac 12 (\alpha -d)(\frac 1p -\frac{1}{p_d})}$ when $\alpha<d$ and $p<p_d$.

We can also extend Theorem \ref{thm-RD-lp-0} to intermediate dimensions. Let $2\leq m\leq d$. Denote $p_m=\frac{2(m+1)}{m-1}$. Let $R^{-1/2}\leq r\leq 1$, and $f=\sum_{T\in \mathbb{W}} f_T$, where each $f_T$ is microlocalized to $(T, \theta(T))$. We say that $f$ has \textbf{$(r,m)$-concentrated frequencies} if there is some $m$-dimensional subspace $V$ such that 
$$
\textrm{Angle} (G(\theta(T)), V) \leq r, \quad \forall T\in \mathbb{W}\,.
$$
Obviously, from the definition, all $f$ trivially has $(1,m)$-concentrated frequencies and also $(r,d)$-concentrated frequencies for any $r>0$.  

\begin{theorem} \label{thm-RD-lp}
 Suppose that $f=\sum_{T\in \mathbb{W}}f_T$, where $\mathbb W\subset \mathbb T$ and each $f_T$ is microlocalized to $(T,\theta(T))$. Let $Y$ be a union of $R^{1/2}$-cubes in $B^d_R$ each of which intersects at most $M$ tubes $T\in \mathbb W$. Then the following refined decoupling inequalities hold.

\noindent (a) Let $2\leq m\leq d$, $p\geq 2$. Suppose that $f$ has $(R^{-1/2},m)$-concentrated frequencies. Then
\begin{equation}\label{eq-RD-lp-R}
    \|f\|_{L^p(Y)} \lesssim_\epsilon R^{\gamma_m(p)+\epsilon} M^{\frac 12 -\frac 1p} \left(\sum_{T\in \mathbb W}\|f_T\|_{L^p}^p\right)^{\frac 1p},
\end{equation}
where $\gamma_m(p)=0$ when $2\leq p\leq p_m$, and $\gamma_m(p)=\frac{m-1}{4}-\frac{m+1}{2p}$ when $p\geq p_m$.

\noindent (b) Let $2\leq m\leq d$, $p\leq p_m$, $\alpha\leq d$. Suppose that $f$ has $(R^{-1/2},m)$-concentrated frequencies. Let $H:Y\to [0,1]$ be a function satisfying that $\int_Q H(x)\,dx \lesssim R^{\alpha/2}$ for any $R^{1/2}$-cube $Q$ in $Y$. Then
\begin{equation}\label{eq-RD-lp-R-alpha}
    \|f\|_{L^p(Y;Hdx)} \lesssim_\epsilon R^{\frac 12 (\alpha -d)(\frac 1p -\frac{1}{p_m})+\epsilon} M^{\frac 12 -\frac 1p} \left(\sum_{T\in \mathbb W}\|f_T\|_{L^p}^p\right)^{\frac 1p}.
\end{equation}

\noindent (c) Let $2\leq m\leq d$, $p_d\leq p\leq p_m$, $\alpha\leq d$. Suppose that $f$ has $(r,m)$-concentrated frequencies, where $R^{-1/2}\leq r\leq 1$. Let $H:Y\to [0,1]$ be a function satisfying that $\int_{Q'} H(x)\,dx \lesssim (\frac 1r)^\alpha$ for any $\frac 1r$-cube $Q'$ in $Y$. Then
\begin{equation}\label{eq-RD-lp-r-alpha}
    \|f\|_{L^p(Y;Hdx)} \lesssim_\epsilon R^\epsilon r^{(d-\alpha)(\frac 1p -\frac{1}{p_m})} (r^2 R)^{\frac{d-1}{4}-\frac{d+1}{2p}} M^{\frac 12 -\frac 1p} \left(\sum_{T\in \mathbb W}\|f_T\|_{L^p}^p\right)^{\frac 1p}.
\end{equation}
\end{theorem}

Note that Theorem \ref{thm-RD-lp-0}(a)(b) is a special case of Theorem \ref{thm-RD-lp}(a)(b) with $m=d$. Morally speaking, Theorem \ref{thm-RD-lp}(a)(b) says that if $G(\theta)$'s are concentrated around a subspace, then one can apply decoupling in a lower dimensional space. Theorem \ref{thm-RD-lp}(c) is obtained by a two-step decoupling process, combining Theorem \ref{thm-RD-lp}(b) to first decouple frequencies into $r$-caps and Theorem \ref{thm-RD-lp-0}(a) to further decouple $r$-caps into smaller $R^{-1/2}$-caps.

\begin{remark}
    (i). In our application to the Falconer distance set problem, the weight function $H$ satisfies the ball condition at all scales $\geq 1$. But in our proof of Theorem \ref{thm-RD-lp}, we only need the ball condition at a single scale: scale $R^{1/2}$ for part (b) and scale $\frac 1r$ for part (c).

    (ii). Though Theorem \ref{thm-RD-lp} is stated for all $\alpha\leq d$, better result exists when $\alpha<d-m$. For $d-m\leq \alpha< d-\frac{m+1}{2}$, there may also be room to further improve Theorem \ref{thm-RD-lp}. We don't explore these directions in the current paper as they will not help in our application to Falconer's distance problem. 
\end{remark}

For large $\alpha$, the gain $R^{\frac 12 (\alpha -d)(\frac 1p -\frac{1}{p_m})}$ when $\alpha<d$ and $p<p_m$ in Theorem \ref{thm-RD-lp}(b) is sharp.

\begin{theorem} \label{thm: eg}
    Let $2\leq m\leq d$, $d-\frac{m+1}{2}\leq \alpha\leq d$. Then there are $f, Y$ and $H$ satisfying the conditions of Theorem \ref{thm-RD-lp}(b) such that for any $p>0$ there holds
    \begin{equation}\label{eq-RD-eg}
    \|f\|_{L^p(Y;Hdx)} \gtrsim R^{\frac 12 (\alpha -d)(\frac 1p -\frac{1}{p_m})} M^{\frac 12 -\frac 1p} \left(\sum_{T\in \mathbb W}\|f_T\|_{L^p}^p\right)^{\frac 1p}.
\end{equation}
In fact, the weight function $H$ here satisfies the ball condition at all scales up to $R^{1/2}$:
$$
\int_{\tilde Q} H(x)dx\lesssim s^\alpha, \quad \forall s\text{-cube } \tilde Q, \,\forall 0<s\leq R^{1/2};
$$
moreover, if in addition $\al\geq m$, then $H$ satisfies the ball condition at all scales up to $R$.
\end{theorem}

\begin{remark}
    In our example for Theorem \ref{thm: eg}, $(M, |\mathbb W|)$ is a fixed special pair of values. 
    We expect it to be a difficult and interesting question to further enquire how sharp the gain $R^{\frac 12 (\alpha -d)(\frac 1p -\frac{1}{p_m})}$ is for other given choices of $(M, |\mathbb W|)$. 
\end{remark}

\subsection{Application to Falconer's distance set problem}

Now let us see a classical question in geometric measure theory introduced by Falconer \cite{falconer1985hausdorff} in the early 80s. Let $E\subset\mathbb{R}^d$ be a compact set, its \emph{distance set} $\Delta(E)$ is defined by
$$
\Delta(E):=\{|x-y|:x,y\in E\}\,.
$$

\begin{conjecture*}\label{conj} \textup{[Falconer]}
Let $d\geq 2$ and $E\subset\mathbb{R}^d$ be a compact set. Then
$$
{\dim_H}(E)> \frac d 2 \Rightarrow |\Delta(E)|>0.
$$
Here $|\cdot|$ denotes the Lebesgue measure and ${\dim_H}(\cdot)$ is the Hausdorff dimension.
\end{conjecture*}

Falconer's conjecture remains open in all dimensions as of today. It has attracted a great amount of attention in the past decades. To name a few landmarks: in 1985, Falconer \cite{falconer1985hausdorff} showed that $|\Delta(E)|>0$ if ${\dim_H}(E)>\frac{d}{2}+\frac{1}{2}$. Bourgain \cite{bourgain1994distance} was the first to lower the threshold $\frac{d}{2}+\frac{1}{2}$ in dimensions $d=2, d=3$ and to use the theory of Fourier restriction in the Falconer problem. The thresholds were further improved by Wolff \cite{wolff1999decay} to $\frac{4}{3}$ in the case $d=2$, and by Erdo\u{g}an \cite{erdogan2005} to $\frac{d}{2}+\frac{1}{3}$ when $d\geq 3$. These records were only very recently rewritten:
\[
\begin{cases}
\frac{5}{4}, &d=2, \quad\qquad\text{(Guth--Iosevich--Ou--Wang \cite{guth2020falconer})}\\
\frac{9}{5}, &d=3, \quad \qquad \text{(Du--Guth--Ou--Wang--Wilson--Zhang \cite{DGOWWZ})}\\
\frac d2+\frac 14+\frac{1}{8d-4}, &d\geq 3, \quad \qquad \text{(Du--Zhang \cite{du2019sharp})}\\
\frac{d}{2}+\frac 14, &d\geq 4 \text{ even}, \quad \text{(Du--Iosevich--Ou--Wang--Zhang \cite{du2021improved})}.
\end{cases}
\]

In this paper, we prove the following result on Falconer's conjecture using weighted refined decoupling estimates. Similar to \cite{guth2020falconer, du2021improved}, we in fact prove a slightly stronger version regarding the pinned distance set.

\begin{theorem}\label{thm: Leb}
Let $d\geq 3$ and $E\subset \mathbb{R}^d$ be a compact set. Suppose that 
\[
{\rm dim}_H(E)>  \begin{cases}\frac{d}{2}+\frac{1}{4}-\frac{1}{8d+4},& d\geq 4,\\ \frac{3}{2}+\frac{1}{4}+\frac{17-12\sqrt{2}}{4},& d=3.\end{cases}
\]Then, there is a point $x\in E$ such that $|\Delta_x(E)|>0$, where
$$
\Delta_x(E):=\{|x-y|:\, y\in E\}.
$$
\end{theorem}

Theorem \ref{thm: Leb} improves the thresholds in \cite{DGOWWZ, du2019sharp, du2021improved} in all dimensions $d\geq 3$. The work \cite{DGOWWZ} uses the polynomial partitioning method developed by Guth \cite{guthPolynomial, guthPolynomialII} and refined Strichartz estimates from \cite{DGL, DGLZ}. The main ingredients in \cite{du2019sharp} are broad-narrow analysis, multilinear refined Strichartz estimates, Bourgain--Demeter's $l^2$ decoupling theorem, and a delicate induction on scales argument. In the current paper, we adapt the good tube/bad tube and refined decoupling method pioneered by \cite{guth2020falconer} for dimension $d = 2$ and continued in \cite{du2021improved} for even dimensions $d$. In both papers, Orponen's radial projection theorem \cite{orponen2018radial} plays a key role. However, the argument does not perform well for odd dimensions $d$, and the result of \cite{du2019sharp} provides a better bound for distance sets. The reason is that Orponen's radial projection theorem only works for sets with dimension $> d-1$, where $d$ is the dimension of the ambient space. To overcome this issue, \cite{du2021improved} projected the set onto a generic $(\frac{d}{2}+1)$-dimensional subspace of $\R^d$ (assuming $d$ is even). While this orthogonal projection trick works well in even dimensions, for odd dimensions we are forced to project to a $(\frac{d+1}{2})$-dimensional subspace instead, which creates some loss. To avoid this loss, a natural approach is to avoid the initial orthogonal projection; but then, we need a radial projection theorem that works for sets of dimension $\le d-1$.

One new ingredient in this paper is a radial projection result, Theorem \ref{conj:threshold}, by the third author \cite{KevinRadialProj}. For each $\delta$-tube $T$, let $r(T)\in [\delta, 1]$ be the thickness of the smallest heavy plate containing $T$ (see Section \ref{sec: outline} for the precise definition). We can remove more bad parts (see Section \ref{sec:bad}) using Theorem \ref{conj:threshold} and give a new threshold \eqref{badthreshold} for bad tubes depending on $r(T)$. To deal with the varying values of $r(T)$, we apply weighted refined decoupling estimates in Theorem \ref{thm-RD-lp}(c). In the case that $r(T)=\delta$, the threshold \eqref{badthreshold} is the same as the one obtained from combining Orponen's radial projection theorem and orthogonal projections; however, Theorem \ref{thm-RD-lp} gives an extra gain when $r(T)$ is small. As $r(T)$ increases, the threshold \eqref{badthreshold} gets much better.

\begin{remark}
   In a companion work \cite{DOKZFalconer}, we provide an alternative proof of Theorem \ref{thm: Leb} (in fact, in \cite{DOKZFalconer} we can establish the dimensional threshold $\frac d2 +\frac 14 -\frac{1}{8d+4}$ in all dimensions $d\geq 3$). Compared with \cite{DOKZFalconer}, in the current paper the construction of the good part $\mu_{1,g}$ is simpler and more intuitive so that the control of the bad part is much easier than that in \cite{DOKZFalconer}. On the other hand, the $L^2$ estimate for the good part is slightly complex, for which we need the new weighted refined decoupling.
\end{remark}

\subsection*{Outline}
In Section \ref{sec: dec}, we prove refined decoupling estimates in Theorem \ref{thm-RD-lp}. In Section \ref{sec: eg}, we present sharp examples for Theorem \ref{thm-RD-lp} in the case of large fractal dimensions to prove Theorem \ref{thm: eg}. In Section \ref{sec: app}, we discuss an application of weighted refined decoupling estimates to Falconer distance set problem - proof of Theorem \ref{thm: Leb}.

\subsection*{Notations.} 
Throughout the article, we write $A\lesssim B$ if $A\leq CB$ for some absolute constant $C$; $A\sim B$ if $A\lesssim B$ and $B\lesssim A$; $A\lesssim_\eps B$ if $A\leq C_\eps B$; $A\lessapprox B$ if $A\leq C_\eps R^\eps B$ for any $\eps>0, R>1$.

For a large parameter $R$, ${\rm RapDec}(R)$ denotes those quantities that are bounded by a huge (absolute) negative power of $R$, i.e. ${\rm RapDec}(R) \leq C_N R^{-N}$ for arbitrarily large $N>0$. Such quantities are negligible in our argument.

For subsets $E_1, E_2 \subset \R^d$, $\dist (E_1,E_2)$ is their Euclidean distance.

For $A \subset X \times Y$ and $x \in X$, define the slice $A|_x = \{ y \in Y : (x, y) \in A \}$. Similar definition for $A|_y$, when $y\in Y$.

We say a measure $\mu$ in $\R^d$ is an $\alpha$-dimensional measure with constant $C_\mu$ if it is a probability measure satisfying that
\begin{equation*}
\mu(B(x,t)) \leq C_\mu t^\alpha,\qquad \forall x\in \mathbb{R}^d,\, \forall t>0.
\end{equation*}

An $(r,m)$-plate $H$ in $\R^d$ is the $r$-neighborhood of an $m$-dimensional affine plane in the cube $[-10,10]^d$. More precisely, 
$$
H=\{z\in[-10,10]^d: \dist(z, P_H) < r\},
$$
where $P_H$ is an $m$-dimensional affine plane, which is called the central plane of $H$. A $C$-scaling of $H$ is 
$$
CH=\{z\in[-10,10]^d: \dist(z, P_H) < Cr\}.
$$

We say that an $(r,m)$-plate $H$ is $\gamma$-concentrated on $\mu$ if $\mu(H) \ge \gamma$.

Let $\cE_{r,m}$ be a set of $(r, m)$-plates with the following properties:
\begin{itemize}
    \item Each $(\frac{r}{2}, m)$-plate intersecting $B(0, 1)$ lies in at least one plate of $\cE_{r,m}$;

    \item For $s \ge r$, every $(s, m)$-plate contains $\lesim \left( \frac{s}{r} \right)^{(m+1)(d-m)}$ many $(r, m)$-plates of $\cE_{r,m}$.
\end{itemize}
For example, when $m = 1$ and $d = 2$, we can simply pick $\sim r^{-1}$ many $r$-tubes in each of an $r$-net of directions. This generalizes to higher $m$ and $d$ via a standard $r$-net argument, see \cite[Section 2.2]{KevinRadialProj} for the details of its construction.

\begin{acknowledgement}
XD is supported by NSF DMS-2107729 (transferred from DMS-1856475), NSF DMS-2237349 and Sloan Research Fellowship. YO is supported by NSF DMS-2142221 and NSF DMS-2055008. KR is supported by a NSF GRFP fellowship. RZ is supported by NSF DMS-2207281 (transferred from DMS-1856541), NSF DMS-2143989 and the Sloan Research Fellowship.
\end{acknowledgement}

\section{Weighted refined decoupling estimates - Proof of Theorem \ref{thm-RD-lp}} \label{sec: dec}
\setcounter{equation}0
In this section, we prove Theorem \ref{thm-RD-lp} for the truncated paraboloid. The proof can be generalized for any compact and strictly convex $C^2$ hypersurface with Gaussian curvature $\sim 1$ by standard arguments. In particular, in the proof of part (a), we use the fact that under the assumption of $(R^{-1/2}, m)$-concentration, one can apply Bourgain--Demeter's $l^2$-decoupling in dimension $m$. See \cite[Lemma 9.3]{guthPolynomialII} and \cite[Lemma 7.4]{DGL} for justification of this fact in the case of the truncated paraboloid, and one can follow \cite[Section 7]{BDdecoupling} to generalize this fact to hypersurfaces as in the above.

First, we present a slightly simplified proof of Theorem \ref{thm-RD-lp-0}(a) based on that in \cite{guth2020falconer}.

\vspace{20pt}

\emph{Proof of Theorem \ref{thm-RD-lp-0}(a).} 
Without loss of generality, we can assume that 
\begin{equation} \label{Qconst} \|  f \|_{L^p(Q)} \sim \textrm{ constant for all $R^{1/2}$-cubes $Q \subset Y$}. \end{equation}
	
Now we decompose $f$ as follows. We cover $\mathcal{S}$ with larger blocks $\tau$ of dimensions $R^{-1/4} \times \cdots \times R^{-1/4} \times R^{-1/2}$.  
For each $\tau$ we cover $B^d_R$ with cylinders $\Box$ with radius $R^{3/4}$ and length $R$, with the long axis perpendicular to $\tau$. Each cylinder $\Box$ is associated to a unique $\tau$, which we denote by $\tau(\Box)$. Then we define
$$\mathbb W_\Box := \{ T \in \mathbb W: \theta(T) \subset \tau(\Box) \textrm{ and } T \cap B_R \subset \Box \} $$
and define $f_\Box := \sum_{T\in \mathbb W_\Box} f_T. $  Note that $\widehat f_\Box$ is essentially supported in $\tau(\Box)$. 

Next, write each $\Box$ as a union of cylinders running parallel to the long axis of $\Box$, with radius $R^{1/2}$ and length $R^{3/4}$. Let $Y_{\Box, M'}$ be the union of those cylinders that each intersect $\sim M'$ of the tubes $T \in \mathbb W_\Box$. 

Now we dyadically pigeonhole $M'$ so that
\begin{equation} \label{decentcube}
    \| f \|_{L^p(Q)} \lessapprox \left\| \sum_{\Box:\, Q \subset Y_{\Box, M'}} f_\Box \right\|_{L^p(Q)} 
\end{equation}
for a fraction $\approx 1$ of $Q \subset Y$. We fix this value of $M'$, and from now on we abbreviate $Y_\Box = Y_{\Box, M'}$.  

Denote the collection of cylinders $\Box$ by $\mathbb B$. We dyadically pigeonhole the cubes $Q \subset Y$ according to the number of $\Box \in \mathbb{B}$ so that $Q \subset Y_\Box$.  We get a subset $Y' \subset Y$ so that for each cube $Q \subset Y'$, $Q \subset Y_\Box$ for $\sim M''$ choices of $\Box \in \mathbb{B}$, and $Q$ obeys \eqref{decentcube}.  Moreover, by dyadic pigeonholing, we have $|Y'| \approx |Y|$.  Since each cube $Q \subset Y$ has approximately equal $L^p$ norm, we also get $\| f \|_{L^p(Y')} \approx \| f \|_{L^p(Y)}$.
	
We also note that
\begin{equation} \label{eq-M}
    M'  M'' \lesssim M,
\end{equation}
because a cube $Q \subset Y'$ belongs to $Y_\Box$ for $\sim M''$ different $\Box$, and if $Q \subset Y_\Box$, then it belongs to $T$ for $\sim M'$ different $T \in \WW_\Box$.

Note that an $R^{1/2}$-cube $Q$ lies in one cylinder $\Box$ associated to each cap $\tau$. So by applying Bourgain--Demeter's $l^2$-decoupling \cite{BDdecoupling} at scale $R^{1/2}$ to the RHS of \eqref{decentcube}, for each $Q\subset Y'$ we get
\begin{equation} \label{decex} 
\| f \|_{L^p(Q)} \lessapprox R^{\frac 12 \gamma_d(p)}\left( \sum_{\Box: Q\subset Y_\Box} \| f_{\Box} \|_{L^p(Q)}^2 \right)^{1/2}. 
\end{equation}

The next ingredient is induction on scales.  After parabolic rescaling, the function $f_\Box$ with the decomposition $f_\Box = \sum_{T \in \WW_\Box} f_T$ on the subset $Y_\Box$ is equivalent to the setup of the theorem at scale $R^{1/2}$ instead of scale $R$.  So by induction on the radius, we get a version of our main inequality for each function $f_\Box$:
\begin{equation}\label{indbox} 
\| f_\Box \|_{L^p(Y_\Box)} \lesssim R^{\frac 12 (\gamma_d(p)+\epsilon)}  (M')^{\frac{1}{2} - \frac{1}{p}} \left(\sum_{T \in \WW_\Box}   \| f_T \|_{L^p}^p \right)^{1/p}. \end{equation}

Now, combining all these ingredients \eqref{eq-M}, \eqref{decex} and \eqref{indbox}, we are ready to estimate $\|f\|_{L^p(Y)}$:

\begin{align*}
    \|f\|^p_{L^p(Y)} & \lessapprox \sum_{Q\subset Y'}\|f\|^p_{L^p(Q)} \\
    &\lessapprox R^{\frac p2 \gamma_d(p)} \sum_{Q\subset Y'} \left( \sum_{\Box: Q\subset Y_\Box} \| f_{\Box} \|_{L^p(Q)}^2 \right)^{p/2}\\
    &\lesssim R^{\frac p2 \gamma_d(p)} (M'')^{\frac p2 -1} \sum_{Q\subset Y'}  \sum_{\Box: Q\subset Y_\Box} \| f_{\Box} \|_{L^p(Q)}^p\\
    &\lesssim R^{\frac p2 \gamma_d(p)} (M'')^{\frac p2 -1}  \sum_\Box \| f_{\Box} \|_{L^p(Y_\Box)}^p \\
    &\lesssim R^{p \gamma_d(p)+\frac{p\epsilon}{2}} (M'M'')^{\frac p2 -1}  \sum_\Box \sum_{T \in \WW_\Box}   \| f_T \|_{L^p}^p \\
    &\lesssim R^{p \gamma_d(p)+\frac{p\epsilon}{2}} M^{\frac p2 -1}  \sum_{T \in \WW}   \| f_T \|_{L^p}^p.
\end{align*}
Taking account of $\lessapprox$ throughout, we get
$$ \|f\|_{L^p(Y)} \lesssim R^{\gamma_d(p)+\frac{3\epsilon}{4}} M^{\frac 12 -\frac 1p} \left(\sum_{T\in \mathbb W}\|f_T\|_{L^p}^p\right)^{\frac 1p}.$$
This closes the induction and finishes the proof of Theorem \ref{thm-RD-lp-0}(a).

\vspace{20pt}

\emph{Proof of Theorem \ref{thm-RD-lp}(a)}. The proof is almost identical to that of Theorem \ref{thm-RD-lp-0}(a). The only difference is that, when applying Bourgain--Demeter's $l^2$-decoupling at scale $R^{1/2}$, one uses the decoupling in dimension $m$ (instead of $d$) because of the $(R^{-1/2},m)$-concentration assumption (see \cite[Lemma 9.3]{guthPolynomialII} and \cite[Lemma 7.4]{DGL} for justifications of similar statements). Also, note that after parabolic rescaling, the function $f_\Box$ with the decomposition $f_\Box = \sum_{T \in \WW_\Box} f_T$ on the subset $Y_\Box$ is equivalent to the setup of the theorem at scale $R^{1/2}$: the $f_\Box$ after rescaling has $(R^{-1/4},m)$-concentrated frequencies.

\vspace{20pt}

\emph{Proof of Theorem \ref{thm-RD-lp}(b)}. Now we prove part (b) using the case $p=p_m$ of part (a). Without loss of generality, we can assume that 
$$
|f_T|\sim \textrm{ constant for all $T\in \WW$},
$$
and each $R^{1/2}$-cube $Q$ in $Y$ intersects $\sim M$ tubes $T\in\WW$. Denote the number of $R^{1/2}$-cubes in $Y$ by $N$, and let $W=|\WW|$. Considering the incidence between $R^{1/2}$-cubes in $Y$ and tubes $T\in\WW$, we get 
$$
NM\lessapprox WR^{1/2}.
$$
By the assumption that $|f_T|\sim$ constant for all $T\in \WW$, we also get 
\begin{equation}\label{eqn: part c}
\left(\sum_{T\in \mathbb W}\|f_T\|_{L^{p_m}}^{p_m}\right)^{\frac {1}{p_m}}\lessapprox \left(\frac{1}{WR^{(d+1)/2}}\right)^{\frac 1p -\frac{1}{p_m}}
\left(\sum_{T\in \mathbb W}\|f_T\|_{L^p}^p\right)^{\frac 1p}.
\end{equation}

Combining all these ingredients together with the assumption that $\int_Q H(x)\,dx \lesssim R^{\alpha/2}$ for any $R^{1/2}$-cube $Q$ in $Y$, and applying the case $p=p_m$ of part (a) we get the following for any $p\leq p_m$:
\begin{align*}
    \|f\|_{L^p(Y;Hdx)} &\leq \left(\int_Y H\,dx\right)^{\frac 1p -\frac{1}{p_m}} \|f\|_{L^{p_m}(Y)}\\
    &\lessapprox \left(\int_Y H\,dx\right)^{\frac 1p -\frac{1}{p_m}} M^{\frac 12 -\frac{1}{p_m}} \left(\sum_{T\in \mathbb W}\|f_T\|_{L^{p_m}}^{p_m}\right)^{\frac {1}{p_m}}\\
    &\lessapprox M^{\frac 12 -\frac 1p} \left(\frac{NR^{\alpha/2}M}{WR^{(d+1)/2}}\right)^{\frac 1p -\frac{1}{p_m}}
\left(\sum_{T\in \mathbb W}\|f_T\|_{L^p}^p\right)^{\frac 1p}\\
    &\lesssim
    R^{\frac 12 (\alpha -d)(\frac 1p -\frac{1}{p_m})} M^{\frac 12 -\frac 1p} \left(\sum_{T\in \mathbb W}\|f_T\|_{L^p}^p\right)^{\frac 1p},
\end{align*}
as desired. 

\vspace{20pt}

\emph{Proof of Theorem \ref{thm-RD-lp}(c)}. We prove part (c) by combining two steps of refined decoupling inequalities from (b) and Theorem \ref{thm-RD-lp-0}(a).

Let $R^{-1/2}\leq r\leq 1$. Let $\tilde Y$ be a union of $r^{-2}$-cubes in $B_R$ such that each $r^{-2}$-cube $Q_1$ in $\tilde Y$ intersects some $R^{1/2}$-cube $Q$ in $Y$. Without loss of generality, we can assume that 
\begin{equation} \label{Q1const} \|  f \|_{L^p(Q_1;Hdx)} \sim \textrm{ constant for all $r^{-2}$-cubes $Q_1 \subset \tilde Y$},
\end{equation}
and 
\begin{equation} \|  f \|_{L^p(Y;Hdx)} \lessapprox \|  f \|_{L^p(\tilde Y;Hdx)} .
\end{equation}

Now we decompose $f$ as follows. We cover $\mathcal{S}$ with blocks $\tau$ of dimensions $r \times \cdots \times r \times r^{2}$.  
For each $\tau$ we cover $B^d_R$ with cylinders $\Box$ with radius $rR$ and length $R$, with the long axis perpendicular to $\tau$. Each cylinder $\Box$ is associated to a unique $\tau$, which we denote by $\tau(\Box)$. Then we define
$ \mathbb W_\Box := \{ T \in \mathbb W: \theta(T) \subset \tau(\Box) \textrm{ and } T \cap B_R \subset \Box \} $
and define $f_\Box := \sum_{T\in \mathbb W_\Box} f_T. $  Note that $\widehat f_\Box$ is essentially supported in $\tau(\Box)$. Denote the collection of boxes $\Box$ by $\mathbb B$.

Next, write each $\Box$ as a union of cylinders $\Box_1$ running parallel to the long axis of $\Box$, with radius $R^{1/2}$ and length $r^{-1}R^{1/2}$. Let $Y_{\Box, M_2}$ be the union of those cylinders that each intersect $\sim M_2$ of the tubes $T \in \mathbb W_\Box$. 

For each $r^{-2}$-cube $Q_1$ in $\tilde Y$, let 
$$
f|_{Q_1}=\sum_{T_1\in \TT[Q_1]} f_{T_1}
$$
be the wave packet decomposition of $f|_{Q_1}$ at scale $r^{-2}$. Each tube $T_1\in \TT[Q_1]$ has radius roughly $r^{-1}$ and length $r^{-2}$. Note that each $T_1$ is contained in a unique cylinder $\Box_1$ with radius $R^{1/2}$ and length $r^{-1}R^{1/2}$, which runs in the same direction as $T_1$. We denote this $\Box_1$ by $\Box_1(T_1)$. And this $\Box_1(T_1)$ is contained in a unique $\Box \in \mathbb B$, which runs in the same direction as $T_1$. We denote this $\Box$ by $\Box(T_1)$.

Now we dyadically pigeonhole $M_2$ so that
\begin{equation} \label{decentcube1}
    \| f \|_{L^p(Q_1;Hdx)} \lessapprox \left\| \sum_{T_1\in\TT[Q_1]:\, \Box_1(T_1)\subset Y_{\Box(T_1),M_2}} f_{T_1} \right\|_{L^p(Q_1;Hdx)} 
\end{equation}
for a fraction $\approx 1$ of $Q_1 \subset \tilde Y$. We fix this value of $M_2$, and from now on we abbreviate $Y_\Box = Y_{\Box, M_2}$.  

Next, write each $Q_1 \cap Y$ as a union of $r^{-1}$-cubes $Q'$. Let $Y_{Q_1, M_1}$ be the union of those $r^{-1}$-cubes that each intersect $\sim M_1$ of the tubes $T_1 \in \TT[Q_1]$ with $\Box_1(T_1) \subset Y_{\Box(T_1)}$. 

Now we dyadically pigeonhole $M_1$ so that
\begin{equation} \label{decentcube2}
    \| f \|_{L^p(Q_1;Hdx)} \lessapprox \left\| \sum_{T_1\in\TT[Q_1]:\, \Box_1(T_1)\subset Y_{\Box(T_1)}} f_{T_1} \right\|_{L^p(Y_{Q_1,M_1};Hdx)}
\end{equation}
for a fraction $\approx 1$ of $Q_1 \subset \tilde Y$. We fix this value of $M_1$, and from now on we abbreviate $Y_{Q_1} = Y_{Q_1, M_1}$. 

Let $\tilde Y'$ be the collections of $Q_1$ satisfying \eqref{decentcube2}. Since $\|  f \|_{L^p(Q_1;Hdx)} \sim \textrm{ constant }$ for all $Q_1 \subset \tilde Y$, we get 
$\|f\|_{L^p(Y;Hdx)}\approx\| f \|_{L^p(\tilde Y;Hdx)} \approx \| f \|_{L^p(\tilde Y';Hdx)}$.

We also note that
\begin{equation}\label{M}
    M_1  M_2 \lesssim M,
\end{equation}
because an $r^{-1}$ cube $Q' \subset Y_{Q_1}$ intersects $\sim M_1$ of the tubes $T_1 \in \TT[Q_1]$ with $\Box_1(T_1) \subset Y_{\Box(T_1)}$, and each $\Box_1(T_1) \subset Y_{\Box(T_1)}$ intersects $\sim M_2$ of the tubes $T \in \WW_{\Box(T_1)}$.

Let $p_d\leq p\leq p_m$. By applying part (b) to the RHS of \eqref{decentcube2}, for each $r^{-2}$-cube $Q_1$ in $\tilde Y'$, we have
\begin{equation}\label{dec1}
    \| f \|_{L^p(Q_1;Hdx)} \lessapprox r^{(d-\alpha)(\frac 1p -\frac{1}{p_m})} M_1^{\frac 12 -\frac 1p} \left(\sum_{T_1\in\TT[Q_1]:\, \Box_1(T_1)\subset Y_{\Box(T_1)}}\|f_{T_1}\|_{L^p}^p\right)^{\frac 1p}.
\end{equation}

Also, after parabolic rescaling, the function $f_\Box$ with the decomposition $f_\Box = \sum_{T \in \WW_\Box} f_T$ on the subset $Y_\Box$ is equivalent to the setup of Theorem \ref{thm-RD-lp-0}(a) at scale $r^2R$ instead of scale $R$. So by applying Theorem \ref{thm-RD-lp-0}(a) in the case $p\geq p_d$, we get
\begin{equation}\label{dec2} 
\| f_\Box \|_{L^p(Y_\Box)} \lessapprox (r^2 R)^{\frac{d-1}{4}-\frac{d+1}{2p}} M_2^{\frac{1}{2} - \frac{1}{p}} \left(\sum_{T \in \WW_\Box}   \| f_T \|_{L^p}^p \right)^{1/p}. \end{equation}

Now we are ready to estimate $\|f\|_{L^p(Y;Hdx)}$ by combining \eqref{M}, \eqref{dec1} and \eqref{dec2}:
\begin{align*}
    \|f\|^p_{L^p(Y;Hdx)} &\approx \sum_{Q_1\subset \tilde Y'}\| f \|^p_{L^p(Q_1;Hdx)} \\
    &\lessapprox \sum_{Q_1\subset \tilde Y'}r^{(d-\alpha)(1 -\frac{p}{p_m})} M_1^{\frac p2 -1} \left(\sum_{T_1\in\TT[Q_1]:\, \Box_1(T_1)\subset Y_{\Box(T_1)}}\|f_{T_1}\|_{L^p}^p\right) \\
    &\sim r^{(d-\alpha)(1 -\frac{p}{p_m})} M_1^{\frac p2 -1}  \sum_{\Box} \sum_{Q_1\subset \tilde Y'} \sum_{\underset{\Box_1(T_1)\subset Y_{\Box(T_1)},\, \Box(T_1) =\Box}{T_1\in\TT[Q_1]}}\|f_{T_1}\|_{L^p}^p \\
    & \lesssim r^{(d-\alpha)(1 -\frac{p}{p_m})}  M_1^{\frac p2 -1}  \sum_{\Box} \|f_\Box\|^p_{L^p(Y_\Box)}\\
    & \lessapprox r^{(d-\alpha)(1 -\frac{p}{p_m})} (r^2 R)^{\frac{(d-1)p}{4}-\frac{d+1}{2}} (M_1M_2)^{\frac p2 -1} \sum_\Box \left(\sum_{T \in \WW_\Box}   \| f_T \|_{L^p}^p \right) \\
    & \lesssim r^{(d-\alpha)(1 -\frac{p}{p_m})} (r^2 R)^{\frac{(d-1)p}{4}-\frac{d+1}{2}} M^{\frac p2 -1} \left(\sum_{T \in \WW}   \| f_T \|_{L^p}^p \right)\,.
\end{align*}
This concludes the proof of (c).

\section{A sharp example in the case of large fractal dimensions - Proof of Theorem \ref{thm: eg}} \label{sec: eg}
We consider the following example obtained by adapting the one in \cite{BBCRV} to intermediate dimensions. Similar adaptions can be found in \cite{DKWZ, du2019upper}.

Let $c=1/1000$ be a fixed small constant, $0<\ka<1/2$, and $2\leq m\leq d$. Denote
$$x=(x_1,\cdots,x_d)=(x',x'',x_d)\in B^d(0,R)\,,$$ $$\xi=(\xi_1,\cdots,\xi_{d-1})=(\xi',\xi'')\in B^{d-1}(0,1)\,,$$
where $$
x'=(x_1,\cdots,x_{d-m}), \quad
x''=(x_{d-m+1},\cdots,x_{d-1}),$$
$$
\xi'=(\xi_1,\cdots,\xi_{d-m}), \quad
\xi''=(\xi_{d-m+1},\cdots,\xi_{d-1}).$$

For simplicity, we denote $B^d(0,r)$ by $B^d_r$,  and write the interval $(-r,r)$ as $I_r$. Let $g(\xi)=\hichi_\Om(\xi)$, where the set $\Om$ is defined by
\begin{equation} \label{Om}
    \Om:=\left[B^{d-m}_{cR^{-1/2}} \times \left(2\pi R^{-\ka} \ZZ^{m-1}+B^{m-1}_{cR^{-1}}\right)\right]\cap B^{d-1}(0,1)\,.
\end{equation}
Take $f(x)=Eg(x)$, the Fourier extension of $g$ over the truncated paraboloid:
\begin{equation}\label{Ef}
   f(x):=Eg(x)=\frac{1}{(2\pi)^{d/2}} \int_{B^{d-1}(0,1)} e^{i(x'\cdot\xi'+x''\cdot \xi'' +x_d|\xi'|^2+x_d|\xi''|^2)}g(\xi)\,d\xi. 
\end{equation}

Next, we define a set $\La$ in $B^d(0,R)$ by \begin{equation}\label{La}
    \La:=\left[B^{d-m}_{cR^{1/2}}\times \left(R^{\ka}\ZZ^{m-1}+B^{m-1}_{c}\right)\times\left(\frac{1}{2\pi}R^{2\ka}\ZZ+I_{c}\right)\right] \cap B^d_R\,.
\end{equation}
And define $Y$ and $H$ by
\begin{equation} \label{Y H}
    Y:= B^{d-m}_{cR^{1/2}}\times B^m_R
\quad \text{ and } \quad   H:=\hichi_\La.
\end{equation}

From the definition, it follows that
\begin{equation}\label{OmSize}
   |\Om|\sim R^{(\ka-1)(m-1)-(d-m)/2}\,.
\end{equation}
and
\begin{equation}\label{sizeLa}
   |\La|\sim R^{(d-m)/2+(1-\ka)(m-1)+1-2\ka}=R^{(d+m)/2-\ka(m+1)}\,.
\end{equation}
And it is straightforward (for example, see \cite[proof of Lemma 3.1]{du2019upper}) to check
\begin{equation} \label{xix}
    x'\cdot\xi'+x''\cdot \xi'' +x_d|\xi'|^2+x_d|\xi''|^2 \in 2\pi \mathbb Z + (-\frac{1}{100},\frac{1}{100})\,,
\end{equation}
provided that $\xi\in\Om$ and $x\in\La$.

\begin{claim}
    For $d-\frac{m+1}{2}\leq \alpha\leq d$, we can take 
\begin{equation} \label{ka}
    \ka:=\frac{d-\al}{2(m+1)}
\end{equation}
such that $\ka\leq 1/4$, and $f, Y$ and $H$ defined as in the above satisfy the conditions of Theorem \ref{thm-RD-lp}(b), and for any $p>0$ there holds
    \begin{equation}\label{eq-RD-eg'}
    \|f\|_{L^p(Y;Hdx)} \gtrsim R^{\frac 12 (\alpha -d)(\frac 1p -\frac{1}{p_m})} M^{\frac 12 -\frac 1p} \left(\sum_{T\in \mathbb W}\|f_T\|_{L^p}^p\right)^{\frac 1p}.
\end{equation}
\end{claim}

Indeed, from the constructions, we have the following properties:
\begin{itemize}
    \item $f$ has $(R^{-1/2},m)$-concentrated frequencies;
    \item By \eqref{xix}, $|f(x)|\sim |\Om|, \forall x\in \La$;
    \item The Fourier support of $f$ are partitioned into $\sim R^{\ka(m-1)}$ many parabolic caps $\theta$, each of radius $R^{-1/2}$;
    \item For each fixed $\theta$, let $f_\theta:=(\hat f|_\theta)^\vee$. Then, by a reason similar to \eqref{xix}, we get    $|f_\theta|\sim R^{-(d-m)/2-(m-1)}$ on $Y$;
    \item Take $\mathbb W := \{T\in \mathbb T: T\subset Y\}$. Then for each $R^{1/2}$-cube in $Y$, there are $M\sim R^{\ka(m-1)}$ many $T\in \mathbb W$ passing through it;
    \item For each $R^{1/2}$-cube $Q$ in $Y$,
    $$
    \int_Q H(x)\,dx =|Q\cap \La|\sim R^{d/2-(m+1)\ka}=R^{\al/2}\,,
    $$
    by the choice of $\ka$ in \eqref{ka}.   
\end{itemize}

Therefore, 
$$
\|f\|_{L^p(Y;Hdx)} \sim |\Om|\cdot |\La|^{1/p}\,,
$$
and
\[
\begin{split}
    &M^{\frac 12 -\frac 1p} \left(\sum_{T\in \mathbb W}\|f_T\|_{L^p}^p\right)^{\frac 1p} \sim M^{\frac 12 -\frac 1p} \left(\sum_{\theta}\|f_\theta\|_{L^p(Y)}^p\right)^{\frac 1p} \\
    &\sim R^{\frac{\ka(m-1)}{2}}R^{-\frac{d-m}{2}-(m-1)}R^{\frac{d+m}{2p}}\,,
\end{split}
\]
and thus 
$$
\frac{\|f\|_{L^p(Y;Hdx)}}{M^{\frac 12 -\frac 1p} \left(\sum_{T\in \mathbb W}\|f_T\|_{L^p}^p\right)^{\frac 1p}} \sim R^{\ka(\frac{m-1}{2}-\frac{m+1}{p})}=R^{\frac 12 (\alpha -d)(\frac 1p -\frac{1}{p_m})}\,,
$$
as desired.

Moreover, by direct computation, one can verify the ball condition at all scales up to $R^{1/2}$:
$$
\int_{\tilde Q} H(x)dx =|\tilde Q \cap \La | \lesssim s^\alpha, \quad \forall s\text{-cube } \tilde Q \text{ in } Y, \forall 0<s\leq R^{1/2};
$$
and also at all scales up to $R$ if in addition $\al\geq m$.

This completes the proof of Theorem \ref{thm: eg}.

\section{Application to Falconer distance set problem} \label{sec: app}
\setcounter{equation}0

As an application of Theorem \ref{thm-RD-lp}, in this section we give a proof of Theorem \ref{thm: Leb} which is different from that in \cite{DOKZFalconer}.

We first recall the following new radial projection theorem, which follows from \cite[Theorem 1.13]{KevinRadialProj}.
\begin{theorem}\label{conj:threshold}
Let $m \in \{ 1, 2, \cdots, d-1 \}$, $m-1 < \alpha \le m$, and fix $\eta, \eps > 0$, and two $\alpha$-dimensional measures $\mu_1, \mu_2$ with constants $C_{\mu_1}, C_{\mu_2}$ supported on $E_1, E_2 \subset B(0, 1)$ respectively. There exists $\gamma > 0$ depending on $\eta, \eps, \alpha, m$ such that the following holds. Fix $\delta < r < 1$. Let $A$ be the set of pairs $(x, y) \in E_1 \times E_2$ satisfying that $x$ and $y$ lie in some $\delta^\eta$-concentrated $(r,m)$-plate on $\mu_1 + \mu_2$. Then there exists a set $B \subset E_1 \times E_2$ with $\mu_1 \times \mu_2 (B) \le \delta^{\gamma}$ such that for every $x \in E_1$ and $\delta$-tube $T$ through $x$, we have
\begin{equation*}
    \mu_2 (T \setminus (A|_x \cup B|_x)) \lesim \frac{\delta^\alpha}{r^{\alpha-(m-1)}} \delta^{-\eps}.
\end{equation*}
The implicit constant may depend on $\eta, \eps, \alpha, m, C_{\mu_1}, C_{\mu_2}$.
\end{theorem}

Essentially, this theorem says that, up to a small loss, one can assume that the wave packets associated with the fractal measure supported on the set of interest all have small mass. We explore this idea in detail in the next subsections.

\subsection{Outline of the proof of Theorem \ref{thm: Leb}}\label{sec: outline}

The setup of the proof is in the same line as \cite{guth2020falconer,du2021improved}. To begin with, let $E\subset \mathbb{R}^d$ be a compact set with positive $\alpha$-dimensional Hausdorff measure, with $\frac{d}{2}<\alpha<\frac{d+1}{2}$. Without loss of generality, assume that $E$ is contained in the unit ball, and there are subsets $E_1, E_2\subset E$, each with positive $\alpha$-dimensional Hausdorff measure, 
and ${\rm dist}(E_1, E_2)\gtrsim 1$. Then there exist $\alpha$-dimensional probability measures $\mu_1$ and $\mu_2$ supported on $E_1$ and $E_2$ respectively, according to the classical Frostman lemma.

To relate the measures to the distance set, we consider their pushforward measures under the distance map. For a fixed point $x\in E_2$, let $d^x:E_1\to \R$ be the pinned distance map given by $d^x(y):=|x-y|$. Then, the pushforward measure $d^x_\ast(\mu_1)$, defined as
\[
\int_{\mathbb{R}} \psi(t)\,d^x_\ast(\mu_1)(t)=\int_{E_1}\psi(|x-y|)\,d\mu_1(y),
\]
is a natural probability measure that is supported on $\Delta_x(E_1)$.

The idea is that we will construct another complex-valued measure $\mu_{1,g}$ that is the \emph{good} part of $\mu_1$ with respect to $\mu_2$, and study its pushforward under the map $d^x$. To set things up, we recall the following decomposition of a function into microlocalized pieces, which has been used in \cite{guth2020falconer,du2021improved}.  We will eventually be choosing the following small parameters with the dependence
$$
0<\beta\ll \gamma\ll\eta \ll \varepsilon \ll\epsilon.
$$

Fix a large parameter $R_0$, to be determined later, and consider a sequence of scales $R_j=2^j R_0$, $\forall j\geq 1$. In $\mathbb{R}^d$, cover the annulus $R_{j-1} \le |\omega| \le R_j$ by rectangular blocks $\tau$ with dimensions approximately $R_j^{1/2} \times \cdots \times R_j^{1/2} \times R_j$, with the long direction of each block $\tau$ being the radial direction.  Choose a smooth partition of unity subordinate to this cover such that
$$
1 = \psi_0 + \sum_{j \ge 1, \tau} \psi_{j, \tau},
$$
where $\psi_0$ is supported in the ball $B(0,2R_0)$.

Let $\beta > 0$ be a sufficiently small constant that we will choose later (depending on $\eta, \varepsilon$). For each $(j, \tau)$, cover the unit ball in $\mathbb{R}^d$ with tubes $T$ of dimensions approximately $R_j^{-1/2 + \beta} \times\cdots \times R_j^{-1/2+\beta} \times 2$ with the long axis parallel to the long axis of $\tau$. The covering has uniformly bounded overlap, each $T$ intersects at most $C(d)$ other tubes. We denote the collection of all these tubes as $\mathbb{T}_{j, \tau}$. Let $\eta_T$ be a smooth partition of unity subordinate to this covering, so that for each choice of $j$ and $\tau$,  $\sum_{T \in \mathbb{T}_{j, \tau}} \eta_T$ is equal to $1$ on the ball of radius $2$ and each $\eta_T$ is smooth.

For each $T \in \mathbb{T}_{j, \tau}$, define an operator
\[
M_T f := \eta_T (\psi_{j, \tau} \hat f)^{\vee},
\]
which, morally speaking, maps $f$ to the part of it that has Fourier support in $\tau$ and physical support in $T$.  Define also $M_0 f := (\psi_0 \hat f)^{\vee}$.  We denote $\mathbb{T}_j = \cup_{\tau} \mathbb{T}_{j, \tau}$ and $\mathbb{T} = \cup_{j \ge 1} \mathbb{T}_j$. Hence, for any $L^1$ function $f$ supported on the unit ball, one has the decomposition
\[
f = M_0 f + \sum_{T \in \mathbb{T}} M_T f+\text{RapDec}(R_0)\|f\|_{L^1}.
\]

Fix parameters $\eta, \varepsilon>0$ (these parameters will be chosen depending on $\epsilon$ in Proposition \ref{prop: l2}). Let $2T$ denote the concentric dilation of $T$ of twice the radius. For each $j\geq 1$ fixed, let $\delta=2 R_j^{-\frac{1}{2}+\beta}$ and consider a dyadic sequence of scales $r\in [C_0\delta, 1]$. Here, $C_0$ is a large constant depending on the separation of the sets $E_1, E_2$. For each fixed $r$, recall that $\mathcal{E}_{r,m}$ denotes a collection of essentially distinct $(r,m)$-plates such that every $(r/2, m)$-plate is contained in some element of $\mathcal{E}_{r,m}$. We further let $\mathcal{H}_r$ denote the subcollection of all $\delta^\eta$-concentrated $(r,m)$-plates of $\mathcal{E}_{r,m}$ on $\mu_1+\mu_2$. Here, $m\in \mathbb{Z}$ is the unique integer satisfying $m-1<\alpha\leq m$. 

For each tube $T \in \mathbb{T}_{j}$, define
\[
r(T):=\min\{{\rm dyadic }\,\,r\in [C_0 \delta,1]:\, \exists H\in \mathcal{H}_r \,\,{\rm s.t. }\,2T\subset H\}.
\]Here, $r(T)$ captures the most efficient choice of the scale of heavy plate that $T$ is contained in.

We say a tube $T \in \mathbb{T}_{j}$ is \emph{bad} if
\begin{equation} \label{badthreshold}
    \mu_2 (4T) \ge \frac{\delta^{\alpha-2\varepsilon}}{r(T)^{\alpha-(m-1)}}.
\end{equation}
Here, $\varepsilon$ is a fixed parameter and will be chosen in Proposition \ref{prop: l2} below. A tube $T$ is \emph{good} if it is not bad, and we define
\[
\mu_{1,g}:=M_0 \mu_1 + \sum_{T \in \mathbb{T}, T \textrm{ good}} M_T \mu_1.
\]
We point out that $\mu_{1,g}$ is only a complex valued measure, and is essentially supported in the $R_0^{-1/2+\beta}$-neighborhood of $E_1$ with a rapidly decaying tail away from it.

Theorem \ref{thm: Leb} will follow from the following two main estimates in the exact same way as in \cite{du2021improved}. We omit the details.

\begin{prop} \label{mainest1} Let $d\geq 3$, and $0<\alpha \leq d-1$. There exists a choice of $\beta>0$ and sufficiently large $R_0$ in the construction of $\mu_{1,g}$ in the above, such that there is a subset $E_2' \subset E_2$ so that $\mu_2(E_2') \ge 1 - \frac{1}{1000}$ and for each $x \in E_2'$,
$$	
	\| d^x_*(\mu_1) - d^x_*(\mu_{1,g}) \|_{L^1} < \frac{1}{1000}.
$$
\end{prop}

\begin{prop} \label{mainest2} Let $d\geq 3$ and
\[
\alpha > \begin{cases}\frac{d}{2}+\frac{1}{4}-\frac{1}{4(2d+1)},& d\geq 4,\\ \frac{3}{2}+\frac{1}{4}+\frac{17-12\sqrt{2}}{4},& d=3,\end{cases}
\]then there exist choices of $\varepsilon$ and $R_0$ so that for sufficiently small $\beta$ in terms of $\alpha$ in the construction of $\mu_{1,g}$ in the above,
$$	
\int_{E_2}  \| d^x_*(\mu_{1,g}) \|_{L^2}^2 d \mu_2(x) < + \infty.
$$
\end{prop}

We will prove these two propositions in the next two subsections. Before that, let us briefly explain the key new ideas here. In contrast to \cite{guth2020falconer,du2021improved}, where a similar framework were used to study the Falconer distance problem, our definition of good tubes here involves a new parameter $r(T)$, which captures the size of the smallest heavy plate containing tube $T$. Thanks to the new radial projection result (Theorem \ref{conj:threshold}), we are able to show that the bad tubes can be safely removed (Proposition \ref{mainest1}). To prove Proposition \ref{mainest2}, we make use of the weighted refined decoupling estimates to deal with varying values of $r(T)$.

In the extreme case that $r(T)\sim R_{j}^{-1/2}$, the threshold for the good/bad tubes is the same as the one in \cite{du2021improved}, however, in this case, from the fact that $T$ is contained in a thin $(R_{j}^{-1/2},m)$-plate, we get extra gain by applying the weighted refined decoupling estimate in Theorem \ref{thm-RD-lp}. At the other end of the spectrum, where $r(T)=1$, even though the weighted decoupling estimate is not going to help, the threshold for the good/bad tubes becomes much better than the one in \cite{du2021improved}, allowing us to obtain improvement for the Falconer problem in this case as well.

\subsection{Removal of the bad region via the radial projection theorem}

In this section, we apply the radial projection theorem (Theorem \ref{conj:threshold}) to prove Proposition \ref{mainest1} \label{sec:bad}.

From the exact same deduction as in the proof of \cite[Proposition 2.1]{du2021improved} (more precisely, Lemma 3.1), one has the bound
\[
\| d^x_*(\mu_1) - d^x_*(\mu_{1,g}) \|_{L^1}\lesssim \sum_{j\geq 1}R_j^{\beta d}\mu_1({\rm Bad}_j(x))+{\rm RapDec}(R_0),
\]where
\[
{\rm Bad_j}(x):=\bigcup_{T \in \mathbb{T}_j:\, x \in 2T \textrm{ and $T$ is bad}} 2T,\quad \forall j\geq 1.
\]We also denote
\[
\begin{split}
\text{Bad}_j:=&\{(y,x)\in E_1\times E_2:\, y\in \text{Bad}_j(x)\}\\
=&\{(y,x)\in E_1\times E_2:\, \exists \text{ bad }T\in \mathbb{T}_j \text{ s.t. }x,y\in 2T\}.
\end{split}
\]
The goal is then to obtain decay for $\mu_1({\rm Bad}_j(x))$, $\forall j\geq 1$. We have the following estimate, from which Proposition \ref{mainest1} follows.

\begin{lemma}\label{lem: bad}
There exist sufficiently large $R_0$ and sufficiently small $\beta>0$ such that there is a subset $E_2' \subset E_2$ so that $\mu_2(E_2') \ge 1 - \frac{1}{1000}$ and for each $x \in E_2'$,
\[
\mu_1({\rm Bad}_j(x))\lesssim R_j^{-100\beta d}, \quad \forall j\geq 1.\]
\end{lemma}

\begin{proof}
Fix a scale $j\geq 1$. Our goal is to show that
\begin{equation}\label{eqn: mu1mu2}
\mu_1\times \mu_2(\text{Bad}_j)\lesssim R_j^{-200\beta d}.
\end{equation}

To see how this would imply the desired estimate, one first rewrites  
\[
\mu_1\times \mu_2(\text{Bad}_j)=\int \mu_1(\text{Bad}_j(x))\,d\mu_2(x).
\]Then, one can find a set $F_j \subset E_2$ with $\mu_2 (F_j) \le R_j^{-50 \beta d}$ such that $\mu_1 (\text{Bad}_j (x)) \le R_j^{-150 \beta d}$ for $x\in E_2\setminus F_j$. We take $E_2' := E_2 \setminus \bigcup_{j \ge 1} F_j$. Observe that $\mu_2 (E_2') \ge \mu(E_2) -  \sum_{j \ge 1} R_j^{-50\beta d} > 1-R_0^{-\beta}>1-\frac{1}{1000}$ if $R_0$ is sufficiently large, hence the desired result would follow.

To prove (\ref{eqn: mu1mu2}), one recalls that $\delta=2 R_j^{-\frac{1}{2}+\beta}$. For each bad $T\in \mathbb{T}_j$, let $r(T)$ be the parameter as in the definition of bad tubes. Since there are only $\sim \log \frac 1 \delta \sim \log R_j$ many choices of $r(T)$, one can assume that $r(T)$ are the same for bad $T\in \mathbb{T}_j$.

One first decomposes 
\[
\text{Bad}_j=\text{Bad}^1_j\cup \text{Bad}^2_j,
\]where for $k=1,2$,
\[
\text{Bad}^k_j:=\{(y,x)\in E_1\times E_2:\, \exists T\in \mathcal{T}_j^k \text{ s.t. }x,y\in 2T\},
\]with
\[
\mathcal{T}_j^1:=\{T\in \mathbb{T}_j:\, T \text{ is bad},\, C_0\delta\leq r(T)\leq 10C_0\delta\},
\]
\[
\mathcal{T}_j^2:=\{T\in \mathbb{T}_j:\, T \text{ is bad},\,  r(T)> 10C_0\delta\}.
\]

The part $\text{Bad}_j^1$ is simpler and doesn't require Theorem \ref{conj:threshold}. Indeed, all the bad tubes in this case satisfies that $\mu_2(4T)\gtrsim \delta^{m-1-2\varepsilon}$. Since $m-1<\alpha\leq m$, one can follow the exact same argument in \cite[justification of (3.1)]{du2021improved} to project the sets $E_1, E_2$ onto an $m$-dimensional subspace and apply the radial projection theorem of Orponen there. We omit the details. 

Next, to estimate $\text{Bad}_j^2(y)$, define $r=\frac{r(T)}{8}$ and apply the new radial projection theorem (Theorem \ref{conj:threshold}) with the $\delta$ and $r$ specified in the above, and with the roles of $x,y$ interchanged. Note that the parameter $\eta$ is the same as the one in the definition of $r(T)$.

Then, one has a set $B\in E_1\times E_2$ and $\gamma>0$ such that $\mu_1\times \mu_2(B)\leq \delta^\gamma$ and that for each $y\in E_1$ and $\delta$-tube $\tilde{T}$ through $y$, one has
\[
\mu_2(\tilde{T}\setminus (A|_y \cup B|_y))\lesssim \frac{\delta^\alpha}{r^{\alpha-(m-1)}}\delta^{-\varepsilon}.
\]Recall that here
\[
\begin{split}
A:=\{(y,x)\in &E_1\times E_2:\,\\
&x,y \in H,\, \text{ for some $\delta^\eta$-concentrated $(r,m)$-plate $H$}\}.
\end{split}
\]

For any fixed $y\in E_1$ and $T\in \mathcal{T}^2_j(y):=\{T\in \mathcal{T}_j^2:\, y\in 2T\}$, we claim that $2T \cap A|_y=\emptyset$. Indeed, suppose there is a point $x\in A|_y$ such that $x,y\in 2T$. Then, one has that $x,y \in H$ for some $\delta^\eta$-concentrated $(r,m)$-plate $H$. Since $y\in E_1$ and $x\in E_2$ are separated, for sufficiently large constant $C_0$ (depending on ${\rm dist}(E_1, E_2)$), one has that $2T\subset 2H$. By the construction of collections $\{\mathcal{E}_{r,m}\}_r$, $2H$ is contained in some $H'\in \mathcal{E}_{4r,m}$. Obviously, $H'$ is $\delta^\eta$-concentrated, hence one concludes that $2T$ is contained in some $H'\in \mathcal{H}_{\frac{r(T)}{2}}$, which contradicts the definition of $r(T)$. (Note that the assumption $r(T)>10C_0\delta$ in this part is to guarantee that $\frac{r(T)}{2}\geq C_0 \delta$.)

By taking $\beta$ sufficiently small depending on $\gamma$, it suffices to estimate $\mu_1\times \mu_2(\text{Bad}^2_j\setminus B)$. Write 
\[
\mu_1\times \mu_2(\text{Bad}^2_j\setminus B)=\int \mu_2(\text{Bad}^2_j(y)\setminus (B|_y))\,d\mu_1(y)
\]where
\[\text{Bad}^2_j(y):=\{x\in E_2:\, (y,x)\in \text{Bad}^2_j\},\quad\forall y\in E_1.
\]Hence
\[
\text{Bad}^2_j(y)\setminus (B|_y)=\bigcup_{T\in \mathcal{T}_j^2(y)} 2T\setminus B|_y.
\]

For any $y\in E_1$ fixed, since each bad tube $T$ satisfies $\mu_2(4T)\geq \frac{\delta^{\alpha-2\varepsilon}}{r(T)^{\alpha-(m-1)}}$ and each point in $E_2$ is contained in at most $\sim R_j^{d\beta}$ many $4T$ with $y\in 2T$, one has that 
\[
|\mathcal{T}_j(y)|\frac{\delta^{\alpha-2\varepsilon}}{r(T)^{\alpha-(m-1)}}\leq \int \sum_{T\in \mathcal{T}_j(y)} \chi_{4T} \,d\mu_2\lesssim R_j^{d\beta}\int \chi_{\bigcup_{T\in \mathcal{T}_j(y)} 4T}\,d\mu_2\lesssim R_j^{\beta d},
\]hence
\[
|\mathcal{T}_j(y)|\lesssim R_j^{\beta d}\frac{r(T)^{\alpha-(m-1)}}{\delta^{\alpha-2\varepsilon}}.
\]We now have the estimate
\[
\mu_2({\rm Bad}_j^2(y)\setminus (B|_y))
\lesssim \frac{\delta^\alpha}{r^{\alpha-(m-1)}}\delta^{-\varepsilon}|\mathcal{T}_j(y)|\lesssim R_j^{\beta d}\delta^{\varepsilon}.
\]By choosing $\beta$ sufficiently small depending on $\varepsilon$, one has that the above is $\lesssim \delta^{1000\beta d}\lesssim R_j^{-200 \beta d}$. The proof is complete by integrating this estimate in $y$.

\end{proof}

\subsection{$L^2$ bound of the good part via weighted decoupling estimates}

In this section, we prove Proposition \ref{mainest2}, applying weighted refined decoupling estimates. More precisely, we will apply Theorem \ref{thm-RD-lp}(c). The argument below proceeds very similarly as in \cite{guth2020falconer,du2021improved}, except that we need to reduce to the situation where the function has concentrated frequencies in order to apply Theorem \ref{thm-RD-lp}.

For $R>0$, let $\sigma_R$ denote the normalized surface measure on the sphere of radius $R$ in $\mathbb{R}^d$. Our main estimate of the section is the following.
\begin{prop}\label{prop: l2}
Let $\alpha>0$ and $R>10 R_0$, and let $\mu_{1,g}$ be as defined in the above. Then, for all $\epsilon>0$, and $\beta, \varepsilon$ sufficiently small depending on $\alpha, \epsilon$, there holds
\begin{equation}\label{eqn: good L2}
\int |\mu_{1,g}*\hat\sigma_R(x)|^2\,d\mu_2(x)
\lesssim  R^{-\Gamma(d,\alpha)+\epsilon} R^{-(d-1)}\int |\hat\mu_1|^2\psi_R\, d\xi+{\rm RapDec}(R),
\end{equation}where
\[
\Gamma(d,\alpha)=\begin{cases} \frac{d\alpha}{d+1},& d\geq 4,\\ \alpha-\frac{\alpha^2}{6},& d=3,\end{cases}
\]and $\psi_R$ is a weight function which is $\sim 1$ on the annulus $R-1 \le |\xi| \le R+1$ and decays off of it.  To be precise, we could take
$$
\psi_R(\xi) = \left( 1 + | R- |\xi|| \right)^{-100}.
$$
\end{prop}

It is a routine argument to check that Proposition \ref{prop: l2} implies Proposition \ref{mainest2}, for instance see \cite[Proof of Proposition 2.2]{du2021improved}. We include it here for the sake of completeness.

Observe first that
$$
d^x_*(\mu_{1,g})(t)=t^{d-1}\mu_{1,g}*\sigma_t(x)\,.
$$
Since $\mu_{1,g}$ is essentially supported in the $R_0^{-1/2+\beta}$-neighborhood of $E_1$, for $x\in E_2$, we only need to consider $t\sim 1$. Hence, up to a loss of $\text{RapDec}(R_0)$ which is negligible in our argument, we have
\[
\begin{split}
\int_{E_2} \|d_*^x(\mu_{1,g})\|_{L^2}^2\,d\mu_2(x)\lesssim &\int_0^\infty \int_{E_2}|\mu_{1,g}*\sigma_t(x)|^2\,d\mu_2(x)t^{d-1}\,dt\\
\sim & \int_0^\infty \int_{E_2}|\mu_{1,g}*\hat{\sigma}_R(x)|^2\,d\mu_2(x)R^{d-1}\,dR,
\end{split}
\]
where in the second step, we have used a limiting process and an $L^2$-identity proved by Liu \cite[Theorem 1.9]{LiuL2}.

For $R\leq 10 R_0$, we use a simple estimate following from orthogonality (for example, see \cite[Lemma 5.2]{DOKZFalconer}):
\begin{equation*}
        \int_{E_2} |\mu_{1,g} * \hat{\sigma}_R (x)|^2 \, d\mu_2 (x) \lesim (R+1)^{d-1} R^{-(d-1)} \,.
    \end{equation*}
So the small $R$ contribution to $\int \norm{d_*^x (\mu_{1,g})}^2_{L^2} d\mu_2 (x)$ is
    \begin{align*}
        \int_0^{10R_0} (R+1)^{d-1} \, dR \lesim R_0^d.
    \end{align*}

Applying Proposition \ref{prop: l2} for each $R>10 R_0$ and dropping the rapidly decaying tail as we may, one can bound the large $R$ contribution to $\int \norm{d_*^x (\mu_{1,g})}^2_{L^2} d\mu_2 (x)$ by
\[
\begin{split}
&\lesssim \int_{10 R_0}^\infty \int_{\mathbb{R}^d}  R^{-\Gamma(d,\alpha)+\epsilon} \psi_R(\xi) | \hat \mu_1(\xi)|^2 \,d \xi dR  \\
&\lesssim  \int_{\mathbb{R}^d} |\xi|^{-\Gamma(d,\alpha)+\epsilon} | \hat \mu_1 (\xi)|^2 \,d \xi \sim I_{\beta} (\mu_1),
\end{split}
\]where $\beta=d-\Gamma(d,\alpha)+\epsilon$, by a Fourier representation for $I_\beta(\mu_1)$, the $\beta$-dimensional energy of $\mu_1$. One thus has $I_\beta(\mu_1)<\infty$ if $\beta<\alpha$, which is equivalent to 
\[
\alpha > \begin{cases}\frac{d}{2}+\frac{1}{4}-\frac{1}{4(2d+1)},& d\geq 4,\\ \frac{3}{2}+\frac{1}{4}+\frac{17-12\sqrt{2}}{4},& d=3.\end{cases}
\]The proof of Proposition \ref{mainest2} is thus complete.

We are now ready to prove Proposition \ref{prop: l2}. First, let's recall the following \cite[Lemma 7.5]{KevinRadialProj}, which controls the total number of essentially distinct concentrated plates from a fixed scale.

\begin{lemma}\label{lem:few_large_plates}
    Let $m-1 < \alpha \le m$. There is $N=N(\alpha, m)$ such that the following holds: let $\nu$ be an $\alpha$-dimensional measure with constant $C_\nu \ge 1$ and let $\mathcal{H}=\{H\in\cE_{r,m}: \nu(H) \ge a\}$. Then $|\mathcal{H}| \lesim (\frac{C_\nu}{a})^N$.
\end{lemma}

\begin{proof}[Proof of Proposition \ref{prop: l2}]
Fix $\epsilon>0$ and $R>10R_0$. By definition, 
\[
\mu_{1,g}* \hat\sigma_R=\sum_{R_j \sim R} \sum_\tau \sum_{T\in \mathbb{T}_{j, \tau}:\, T \textrm{ good}} M_T \mu_1 * \hat \sigma_R+{\rm RapDec}(R).
\]Since the contribution of ${\rm RapDec}(R)$ is already taken into account in the statement of Proposition \ref{prop: l2}, we may ignore the tail ${\rm RapDec}(R)$ in the argument below. The following reduction is the same as in \cite{du2021improved}. 

Let $\eta_1$ be a bump function adapted to the unit ball and define
\[
f_T = \eta_1 \left( M_T \mu_1 * \hat \sigma_R \right).
\]It is easy to see that $f_T$ is microlocalized to $(T,\theta(T))$ (similarly as defined in Section \ref{sec: dec}, with $f$ essentially supported in $2T$ and $\hat{f}$ essentially supported in $2\theta(T)$, the block of dimension $R^{1/2}\times \cdots \times R^{1/2}\times 1$ in the partition of the $1$-neighborhood of $RS^{d-1}$ corresponding to the long direction of $T$).

We now apply a series of dyadic pigeonholing to the integral to be estimated. First, there exists $\lambda>0$ such that
\[
\int |\mu_{1,g}*\hat\sigma_R(x)|^2\,d\mu_2(x)\lesssim \log R \int | f_\lambda (x) |^2 d\mu_2(x),
\]where
\[
f_\lambda=\sum_{T\in \mathbb{W}_\lambda}f_T,\quad \mathbb{W}_\lambda:= \bigcup_{R_j\sim R}\bigcup_{\tau}\Big\{ T\in \mathbb{T}_{j,\tau}: T \text{ good }, \| f_T \|_{L^p} \sim \lambda \Big\}.
\]

For each $R$, since there are only finitely many $R_j\sim R$, for the sake of brevity we will drop the first union in $\mathbb{W}_\lambda$ and assume without loss of generality that $R_j=R$ for all $T\in \mathbb{W}_\lambda$ from this point on.

In addition, considering a sequence of dyadic scales $r\in [2C_0R^{-\frac{1}{2}+\beta},1]$, by dyadic pigeonholing, one can reduce to a subcollection $\mathbb{W}_{\lambda, r}$ for a fixed $r$, where  
\[
\mathbb{W}_{\lambda, r}:=\{T\in \mathbb{W}_\lambda:\, r(T)= r\}.
\]Here, recall that 
\[
r(T):=\min\{{\rm dyadic }\,\,r\in [2C_0 R^{-\frac{1}{2}+\beta},1]:\, \exists H\in \mathcal{H}_r \,\,{\rm s.t. }\,2T\subset H\},
\]where $\mathcal{H}_r$ denotes a collection of $(2R^{-\frac{1}{2}+\beta})^\eta$-concentrated $(r,m)$-plates of $\cE_{r,m}$.

Fixing such an $r$ from now on, one can further conclude from Lemma \ref{lem:few_large_plates} that there are at most $\sim R^{\frac{N\eta}{2}}$ different $H\in \mathcal{H}_r$. Therefore, choosing $\eta \ll \frac{1}{N}$ (depending on $\varepsilon$), one can assume that there is a fixed $(r,m)$-plate $H\in \mathcal{H}_r$ such that
\[
\int |\mu_{1,g}*\hat\sigma_R(x)|^2\,d\mu_2(x)\lessapprox \int | f_{\lambda,r,H} (x) |^2 d\mu_2(x),
\]where
\[
f_{\lambda, r,H}=\sum_{T\in \mathbb{W}_{\lambda,r,H}}f_T,\quad \mathbb{W}_{\lambda,r,H}:= \{ T\in \mathbb{W}_{\lambda,r}: 2T \subset H \}.
\]

By dyadic pigeonholing again, one is able to find $M>0$ such that the right hand side of the inequality above is further dominated by
\[
\lessapprox\int_{Y_M} | f_{\lambda,r,H} (x) |^2 d\mu_2(x),
\]where the region $Y_M=\bigcup_{q\in \mathcal{Q}_M}q$ is contained in the unit ball and
\[
\mathcal{Q}_M:=\{ R^{-1/2}\textrm{-cube } q: q \textrm{ intersects } \sim M \textrm{ tubes } T \in \mathbb{W}_{\lambda, r, H} \}.
\]

Fix $p\in [p_d, p_m]$ and write $f=f_{\lambda,r,H}$, $Y=Y_M$, and $\mathbb{W}=\mathbb{W}_{\lambda, r, H}$ for the sake of brevity. We first claim that there holds the following $L^p$ estimate:
\begin{equation}\label{eqn: L2mu}
\begin{split}&\|f\|_{L^p(Y;\,\mu_2)}\\ \lessapprox &r^{\left(\frac{1}{p}-\frac{1}{p_m}\right)(d-\alpha)}(r^2R)^{\frac{d-1}{4}-\frac{d+1}{2p}}R^{\frac{d-\alpha}{p}}\left(\frac{M}{|\mathbb{W}|}\right)^{\frac{1}{2}-\frac{1}{p}}\left(\sum_{T\in \mathbb{W}}\|f_T\|^2_{L^p}\right)^{1/2}.
\end{split}
\end{equation}

This is essentially a rescaled version of our new decoupling estimate, more precisely, Theorem \ref{thm-RD-lp}(c). We will leave the justification of it to the end of the proof and proceed now assuming it holds true.

Since $f$ only involves good wave packets, by considering the quantity
$$
\sum_{q\in \mathcal{Q}_M} \sum_{T\in \mathbb{W}: T\cap q\neq \emptyset} \mu_2(q),
$$
one obtains that
\begin{equation}\label{eqn: count}
    M \mu_2 (Y) \lesssim  | \mathbb{W} | \frac{R^{(-\frac{1}{2}+\beta)(\alpha-2\varepsilon)}}{r^{\alpha-(m-1)}}.
\end{equation}
The quantity on the right hand side above follows from the threshold defining good wave packets.

Hence, combining (\ref{eqn: L2mu}) and H\"older's inequality, one has that
\[
\begin{split}
&\int |\mu_{1,g}*\hat\sigma_R(x)|^2\,d\mu_2(x)\lessapprox \int_Y |f(x)|^2\,d\mu_2(x)\\
\leq & \left(\int_Y |f(x)|^p\,d\mu_2(x)\right)^{\frac{2}{p}}\mu_2(Y)^{1-\frac{2}{p}}\\
\lessapprox & r^{2\left(\frac{1}{p}-\frac{1}{p_m}\right)(d-\alpha)}(r^2R)^{\frac{d-1}{2}-\frac{d+1}{p}}R^{\frac{2(d-\alpha)}{p}}\left(\frac{R^{(-\frac{1}{2}+\beta)(\alpha-2\varepsilon)}}{r^{\alpha-(m-1)}} \right)^{1-\frac{2}{p}}\sum_{T\in \mathbb{W}}\|f_T\|^2_{L^p}\\
=& r^{2\left(\frac{1}{p}-\frac{1}{p_m}\right)(d-\alpha)+d-1-\frac{2(d+1)}{p}-\left(1-\frac{2}{p}\right)(\alpha-m+1)}R^{(d-1-\alpha) (\frac{1}{2}+\frac{1}{p})+O_{\alpha}(\beta)+\varepsilon}\sum_{T\in \mathbb{W}}\|f_T\|^2_{L^p}.
\end{split}
\]From standard computation (see for instance \cite[Proof of Lemma 4.1]{du2021improved}), one has the simple bound
\[
\begin{split}
\|f_T\|_{L^p}\lesssim & \|f_T\|_{L^\infty}|T|^{1/p}\lesssim \sigma_R(\theta(T))^{1/2}|T|^{1/p} \|\widehat{M_T\mu_1}\|_{L^2(d\sigma_R)}\\
= & R^{-(\frac{1}{2p}+\frac{1}{4})(d-1)+O_\alpha(\beta)}\|\widehat{M_T\mu_1}\|_{L^2(d\sigma_R)}.
\end{split}
\]Plugging this back into the above and applying orthogonality, one has that
\[
\begin{split}
&\int |\mu_{1,g}*\hat\sigma_R(x)|^2\,d\mu_2(x)\\
\lessapprox &
r^{2\left(\frac{1}{p}-\frac{1}{p_m}\right)(d-\alpha)+d-1-\frac{2(d+1)}{p}-\left(1-\frac{2}{p}\right)(\alpha-m+1)}\cdot \\
&\qquad\qquad\qquad R^{-\alpha (\frac{1}{2}+\frac{1}{p})+O_{\alpha}(\beta)+\varepsilon} \sum_{T\in \mathbb{W}} \|\widehat{M_T\mu_1}\|^2_{L^2(d\sigma_R)}\\
\lesssim &
r^{\frac{2(d-\alpha)}{m+1}+m-2-\frac{2m}{p}}R^{-\alpha (\frac{1}{2}+\frac{1}{p})+O_{\alpha}(\beta)+\varepsilon} R^{-(d-1)}\int |\hat\mu_1|^2\psi_R\, d\xi.
\end{split} 
\]
Thus, for sufficiently small $\beta, \varepsilon$ depending on $\alpha$ and $\epsilon$,
\[
\begin{split}
&\int |\mu_{1,g}*\hat\sigma_R(x)|^2\,d\mu_2(x) \\
\lesssim &
r^{\frac{2(d-\alpha)}{m+1}+m-2-\frac{2m}{p}}R^{-\alpha (\frac{1}{2}+\frac{1}{p})+\epsilon} R^{-(d-1)}\int |\hat\mu_1|^2\psi_R\, d\xi.
\end{split}
\]

Since $r$ can be anything ranging between $R^{-\frac{1}{2}+\beta}$ and $1$, one needs to maximize the right hand side of the above estimate over all possible values of $r$. To do this, we first rewrite $r=R^{-\gamma}$, where $0\leq \gamma\leq \frac{1}{2}$. Then, the estimate above becomes
\begin{equation}\label{eqn: L2gamma}
\begin{split}
&\int |\mu_{1,g}*\hat\sigma_R(x)|^2\,d\mu_2(x)\\
\lessapprox &
R^{-\gamma\left(\frac{2(d-\alpha)}{m+1}+m-2-\frac{2m}{p}\right)-\alpha (\frac{1}{2}+\frac{1}{p})+\epsilon} R^{-(d-1)}\int |\hat\mu_1|^2\psi_R\, d\xi\\
=&R^{-\gamma\left(\frac{2(d-\alpha)}{m+1}+m-2\right)-\frac{\alpha}{2}+\epsilon}R^{\frac{2m\gamma-\alpha}{p}}R^{-(d-1)}\int |\hat\mu_1|^2\psi_R\, d\xi.
\end{split}
\end{equation}

First, we look at the case $d\geq 4$. Note that one can assume without loss of generality that $\frac{d}{2}<\alpha<\frac{d}{2}+\frac{1}{3}$, hence $m=\frac{d+1}{2}$ if $d$ is odd, and $m=\frac{d}{2}+1$ if $d$ is even. A quick calculation then shows that $\frac{2(d-\alpha)}{m+1}+m-2-\frac{2m}{p}>0$ for all $p\in [p_d, p_m]$. Therefore, the estimate is maximized at $\gamma=0$, becoming
\[
\begin{split}
&\int |\mu_{1,g}*\hat\sigma_R(x)|^2\,d\mu_2(x)\\
\lessapprox & R^{-\alpha(\frac{1}{2}+\frac{1}{p})+\epsilon} R^{-(d-1)}\int |\hat\mu_1|^2\psi_R\, d\xi.
\end{split}
\]This bound can then be minimized at $p=p_d$, hence one concludes that
\begin{equation}\label{eqn: high dim}
\int |\mu_{1,g}*\hat\sigma_R(x)|^2\,d\mu_2(x)
\lessapprox  R^{-\frac{d\alpha}{d+1}+\epsilon} R^{-(d-1)}\int |\hat\mu_1|^2\psi_R\, d\xi.
\end{equation}

The computation in the case $d=3$ ($m=2$) is a bit more complicated. Observing (\ref{eqn: L2gamma}), one sees that in the case $\frac{\alpha}{4}\leq \gamma\leq \frac{1}{2}$, the optimal choice of $p$ is $p=p_m=6$, which gives
\[
\int |\mu_{1,g}*\hat\sigma_R(x)|^2\,d\mu_2(x)
\lessapprox  R^{-\gamma(\frac{4}{3}-\frac{2\alpha}{3})-\frac{2\alpha}{3}+\epsilon} R^{-2}\int |\hat\mu_1|^2\psi_R\, d\xi.
\]This bound is the worst when $\gamma=\frac{\alpha}{4}$, at which it becomes 
\begin{equation}\label{eqn: 3 dim}
\int |\mu_{1,g}*\hat\sigma_R(x)|^2\,d\mu_2(x)
\lessapprox  R^{\frac{\alpha^2}{6}-\alpha+\epsilon} R^{-2}\int |\hat\mu_1|^2\psi_R\, d\xi.
\end{equation}In the case $0\leq \gamma\leq \frac{\alpha}{4}$, (\ref{eqn: L2gamma}) is optimized at $p=p_d=4$, which is then maximized at $\gamma=\frac{\alpha}{4}$, producing the same bound as (\ref{eqn: 3 dim}).

In sum, one concludes that the desired estimate (\ref{eqn: good L2}) holds true in all dimensions $d\geq 3$. 

We are left with the justification of estimate (\ref{eqn: L2mu}). 

First, note that $f$ has Fourier support in the $1$-neighborhood of the sphere of radius $R$, one has that
\[
\|f\|_{L^p(Y;\, \mu_2)}^p\lesssim 
\|f\|_{L^p(Y;\, \mu_2\ast \eta_{\frac{1}{R}})}^p=R^{-d}\int_{RY}|f(R^{-1}y)|^p \mu_2\ast \eta_{\frac{1}{R}}(R^{-1}y)\,dy,
\]where $\eta_{\frac{1}{R}}$ is a bump function of integral $1$ that is essentially supported on the ball of radius $\frac{1}{R}$.

The above quantity can be rewritten as $$\sim R^{d-\alpha}R^{-d}\|F\|^p_{L^p(RY;\, Hdy)}\,,$$
where $F(y):=f(R^{-1}y)$ and weight function $H(y):=c_1 R^{\alpha-d}\mu_2\ast \eta_{\frac{1}{R}}(R^{-1}y)$. It is easy to see that Theorem \ref{thm-RD-lp}(c) applies. Indeed, $F$ obviously satisfies all the required conditions. From the observation that $\|\mu_2\ast \eta_{\frac{1}{R}}\|_{L^\infty}\lesssim R^{d-\alpha}$, one can choose a constant $c_1$ such that $H\leq 1$. Moreover, for any ball $Q'$ of radius $\frac{1}{r}$, the $\alpha$-dimensional condition of $\mu_2$ and the fact that $\int \eta_{\frac{1}{R}}=1$ imply that $\int_{Q'} H(y)\,dy\lesssim r^{-\alpha}$, (in fact, we have this ball condition for any radius $t>0$, but we only need this for $t=1/r$ in order to apply Theorem \ref{thm-RD-lp}(c)). 

Therefore, Theorem \ref{thm-RD-lp}(c) yields that
\[
\begin{split}
&\|f\|_{L^p(Y;\mu_2)}\\
\lessapprox &R^{\frac{d-\alpha}{p}}R^{-\frac{d}{p}} r^{(d-\alpha)(\frac 1p -\frac{1}{p_m})} (r^2 R)^{\frac{d-1}{4}-\frac{d+1}{2p}} M^{\frac 12 -\frac 1p} \left(\sum_{T\in \mathbb W}\|F_T\|_{L^p}^p\right)^{\frac 1p}\\
=&R^{\frac{d-\alpha}{p}} r^{(d-\alpha)(\frac 1p -\frac{1}{p_m})} (r^2 R)^{\frac{d-1}{4}-\frac{d+1}{2p}} M^{\frac 12 -\frac 1p} \left(\sum_{T\in \mathbb W}\|f_T\|_{L^p}^p\right)^{\frac 1p},
\end{split}
\]where $F_T(y):=f_T(R^{-1}y)$. Then, recalling that we have pigeonholed at the beginning to reduce to the case that all $\|f_T\|_{L^p}$ are roughly constant, one thus concludes that the right hand side in the above, up to a constant, coincides with the right hand side of the desired inequality (\ref{eqn: L2mu}). The proof is complete.

\end{proof}

\bibliographystyle{plain}
\bibliography{main}

\vspace{0.25cm}
	
\noindent Xiumin Du, Northwestern University, \textit{xdu@northwestern.edu}\\

\noindent Yumeng Ou, University of Pennsylvania, \textit{yumengou@sas.upenn.edu}\\

\noindent Kevin Ren, Princeton University, \textit{kevinren@princeton.edu}\\

\noindent Ruixiang Zhang, UC Berkeley, \textit{ruixiang@berkeley.edu}

\end{document}